\documentclass[10pt]{article}
\usepackage{latexsym}
\usepackage{amssymb}
\usepackage{hhline}
\usepackage{graphicx,amssymb,lineno, amsmath}
\usepackage{multicol}

by 2 \advance \topmargin -1.8in \leftmargin 210mm \advance
\leftmargin -\textwidth \divide \leftmargin by 2 \advance
\leftmargin

\newcommand\T{\rule{0pt}{2.6ex}}
\newcommand\B{\rule[-1.2ex]{0pt}{0pt}}
\newtheorem{theorem}{Theorem}
\newtheorem{lemma}{Lemma}

\newenvironment{proof}{\noindent{\bf Proof:}}{\hspace*{1mm} \hfill$\Box$ \par \bigskip }

\date{}
\title{\bf Computing the Ramsey Number $R(K_5-P_3,K_5)$}

\begin{document}

\thispagestyle{empty}
\maketitle  \footnote[0]{Research supported by the NSF Research Experiences for Undergraduates Program (Award \#0552418) held at the Rochester Institute of Technology during the summer of 2010. The program was cofunded by the Department of Defense.}
\thispagestyle{empty}

\vspace{-18mm}
\begin{center}

\begin{multicols}{2}
{\bf Jesse A. Calvert}\\
 Department of Mathematics\\
 Washington University\\
 St. Louis, MO 63105\\
 {\tt jcalvert@wustl.edu}\\[4mm]

{\bf Michael J. Schuster}\\
 Department of Mathematics\\
 North Carolina State University \\
 Raleigh, NC 27695 \\
 {\tt mjschust@ncsu.edu}\\[4mm]
\end{multicols}

{\bf Stanis\l{}aw P. Radziszowski}\\
 Department of Computer Science \\
 Rochester Institute of Technology \\
 Rochester, NY 14623 \\
 {\tt spr@cs.rit.edu}\\[4mm]

\bigskip
\end{center}

\vspace{-3mm}
\begin{quote}
{\bf Abstract.}  We give a computer-assisted proof of the fact
that $R(K_5-P_3, K_5)=25$. This solves one of the three remaining
open cases in Hendry's table, which listed the Ramsey numbers for
pairs of graphs on 5 vertices. We find that there exist no
$(K_5-P_3,K_5)$-good graphs containing a $K_4$ on 23 or 24 vertices,
where a graph $F$ is $(G,H)$-good if $F$ does not contain $G$ and the
complement of $F$ does not contain $H$.
The unique $(K_5-P_3,K_5)$-good graph
containing a $K_4$ on 22 vertices is presented.
\end{quote}

\section{Introduction}

For simple graphs $G$ and $H$, a $(G,H)$-good graph is a graph $F$ that
contains no subgraph $G$ and whose complement contains no subgraph $H$.
A $(G,H;n)$-good graph is a $(G,H)$-good graph on $n$ vertices. We will
denote the set of all $(G,H)$-good graphs by $\mathcal{R}(G,H)$
and, similarly,
the set of all $(G,H;n)$-good graphs by $\mathcal{R}(G,H;n)$.
The minimum
number of vertices $n$ such that no $(G,H;n)$-good graph exists is
the Ramsey number $R(G,H)$.
The best known bounds for various
types of Ramsey numbers are listed in
the dynamic survey {\em Small Ramsey Numbers} by the third author \cite{Rad}.
For a comprehensive overview of Ramsey numbers and general graph
theory terminology not defined in this paper we recommend
a widely used textbook by West \cite{West}.
$P_k$ denotes a path on $k$ vertices, and $K_5 - P_3$
can be seen either as a $K_5$ with two adjacent edges
removed or a $K_4$ with an additional vertex connected to
two of its vertices.

In 1989, Hendry \cite{Hen} compiled a table of
known values and bounds on Ramsey numbers for
connected graphs $G$ and $H$ on five vertices.
For the Ramsey number $R(K_5-P_3, K_5)$ the Hendry's table
gives the bound $R(K_5-P_3, K_5) \le 28$;
a lower bound of 25 can be obtained from
the result $R(K_4,K_5)=25$ \cite{MR2}. In the 2009 REU
(NSF Research Experiences for Undergraduates Program)
Black, Leven and Radziszowski \cite{Reu}
showed that the upper bound can be reduced
to $R(K_5-P_3, K_5) \le 26$.
The main goal of the 2010 REU was to show that
$R(K_5-P_3, K_5) = 25$ or $R(K_5-P_3, K_5) = 26$,
which was accomplished using a combination of
combinatorial reasoning and computation.
The computations required to show that $R(K_5-P_3, K_5)=25$
were easily completed on a standard desktop computer.
However, the computation of the number
of $(K_5-P_3, K_5)$-good graphs containing $K_4$
on less than $25$ vertices was much longer.
We found that there were no $(K_5-P_3, K_5)$-good
graphs on $24$ or $23$ vertices, exactly one on $22$
vertices, and millions on $21$.

The general question of characterizing graphs $G$, $H$ and
extensions $G'$ of $G$, for which the equality $R(G,H)=R(G',H)$ holds,
is very difficult. Only a few such cases are known,
and some of them are presented in Section 5. Our detailed
study of $(K_5-P_3,K_5)$-good graphs seems to provide
evidence that, at least sometimes, avoiding larger graph
$G'$ may be not much stronger than avoiding $G$.
We expect that many other interesting cases exist
for which $R(G,H)=R(G',H)$.

Section 2 presents two enumerations of smaller graphs needed later in the paper,
the algorithm foundations and computations showing the main results are described
in Sections 3 and 4, respectively, and finally Section 5 points out how our
result relates to a general 1989 theorem by Burr, Erd\H {o}s, Faudree and
Schelp \cite{BEFS}.

\section{Enumerations for $\mathcal{R}(K_5-P_3, K_5)$}

In order to study $\mathcal{R}(K_5-P_3, K_5)$,
it is useful to have enumerations of the sets
$\mathcal{R}(K_4-P_3,K_5)$ and $\mathcal{R}(K_5-P_3,K_4)$.
It is known that $R(K_4-P_3,K_5)=14$ and $R(K_5-P_3,K_4)=18$
\cite{Clan}. We have generated the corresponding sets of graphs
using a simple vertex by vertex extension algorithm, and McKay's
{\em nauty} package \cite{McKay2} to eliminate isomorphs.
The 1092 nonisomorphic graphs in $\mathcal{R}(K_4-P_3,K_5)$
and the 3454499 nonisomorphic graphs in $\mathcal{R}(K_5-P_3,K_4)$
were enumerated. The results agreed with the computations
reported in \cite{Reu}, and the data is summarized in
Tables I and II (two typographical errors in \cite{Reu}
were corrected).
We include the tables here in full since they are needed
to see the context of computations performed to obtain
our results.\\

\vspace{.4in}
\begin{center}
  \begin{tabular}{@{} |r|r r|r r| @{}}
    \hline
$n$ & $|\mathcal{R}(K_5-P_3,K_4;n)|$ & {\#}edges & {\#}graphs with $K_4$ & {\#}edges\\
    \hline
    2 & 2 & 0-1 & 0 & \\
    3 & 4 & 0-3 & 0 & \\ 
    4 & 10 & 1-6 & 1& 6\\ 
    5 & 26 & 2-8 & 2 & 6-7\\ 
    6 & 92 & 3-12 & 8 & 6-12\\ 
    7 & 391 & 5-16 & 29 & 7-12\\ 
    8 & 2228 & 7-21 & 149 & 8-16\\ 
    9 & 15452 & 9-27 & 751 & 10-19\\ 
    10 & 107652 & 12-31 & 3946 & 12-24\\ 
    11 & 557005 & 15-36 & 10649 & 15-28\\ 
    12 & 1455946 & 18-40 & 6780 & 18-32\\ 
    13 & 1184231 & 33-45 & 0 &\\ 
    14 & 130816 & 41-50 & 0 &\\ 
    15 & 640 & 50-55 & 0 &\\ 
    16 & 2 & 60 & 0 &\\ 
    17 & 1 & 68 & 0 &\\ 
    \hline

  \end{tabular}
\end{center}

\begin{center}
{\bf Table I.}
Statistics of $\mathcal{R}(K_5-P_3,K_4)$.  
\end{center}
\vspace{.3in}

The last two columns of Table I give the counts and
the corresponding edge ranges of all $(K_5-P_3, K_4)$-good
graphs which contain $K_4$ as a subgraph, i.e.
of all graphs which are $(K_5-P_3, K_4)$-good
but not $(K_4,K_4)$-good. We will show in Section \ref{sec:thm}
that a similar type of distribution occurs
in $\mathcal{R}(K_5-P_3,K_5)$.

In Table II, the last two columns present counts and
the corresponding edge ranges of all $(K_4-P_3, K_5)$-good
graphs which contain $K_3$ as a subgraph, or equivalently,
those graphs which are $(K_4-P_3, K_5)$-good but not
$(K_3,K_5)$-good.

\vspace{.4in}
\begin{center}
\begin{tabular}{@{} |r|r r| r r|@{}}
    \hline
$n$ & $|\mathcal{R}(K_4-P_3,K_5;n)|$ & {\#}edges & {\#}graphs with $K_3$ & {\#}edges\\
    \hline
    2 & 2 & 0-1 & 0 &\\
    3 & 4 & 0-3 & 1 & 3\\ 
    4 & 8 & 0-4 & 1 & 3\\ 
    5 & 15 & 1-6 & 2 & 3-4\\ 
    6 & 36 & 2-9 & 4 & 3-6\\ 
    7 & 78 & 3-12 & 7 & 4-7\\ 
    8 & 190 & 4-16 & 11 & 5-9\\ 
    9 & 308 & 6-17 & 18 & 6-12\\ 
    10 & 326 & 8-20 & 13 & 8-13\\ 
    11 & 110 & 10-22 & 5 & 10-15 \\ 
    12 & 13 & 12-24 & 1& 12\\ 
    13 & 1 & 26 & 0 & \\ 
    \hline

  \end{tabular}
\end{center}

\begin{center}
{\bf Table II.}
Statistics of $\mathcal{R}(K_4-P_3,K_5)$.
\end{center}
\vspace{.3in}

\section{Properties of $(K_5-P_3,K_5)$-good Graphs} \label{sec:lemma} 

\medskip
Since $R(K_4,K_5)=25$ \cite{MR2}, any $(K_5-P_3,K_5;25)$-good graph $F$
contains at least one $K_4$. Let $x$ be the vertex of this $K_4$ with
the smallest degree.
We denote by $F^+_x$ the graph induced by the neighborhood of
vertex $x$ and by $F^-_x$ the graph induced by the anti-neighborhood
of vertex $x$ (non-neighbors of $x$, not including $x$).
Note that $F^+_x$ must be a $(K_4-P_3,K_5)$-good graph,
while $F^-_x$ must be a $(K_5-P_3,K_4)$-good graph.

\bigskip
\begin{lemma}
For $n \ge 4$, if $F$ is a $(K_5-P_3,K_5;n)$-good
graph containing a $K_4$, then the sum of the degrees
of the vertices in any $K_4$ contained in $F$ is at most $n+8$.
\end{lemma}

\begin{proof}
Let $d_1,d_2,d_3,$ and $d_4$ be the degrees of vertices of
a $K_4$ in $F$. The neighborhoods of each vertex in this $K_4$
must be disjoint, otherwise we have a $K_5-P_3$ subgraph in $F$.
Hence 
\[
\sum_{i=1}^4 (d_i-3)+4 \le n \mbox{, or } \sum_{i=1}^{4} d_i \le n + 8
\]
\end{proof}

\medskip
For a given $n$, we can determine the maximum of the minimum degree
vertex $x$ in the $K_4$ under consideration, and thus the possible
values of $|V(F^+_x)|$ and $|V(F^-_x)|$.
Note that $|V(F^+_x)|+|V(F^-_x)|=n-1$ in each case.
All possibilities for $n \ge 22$ are summarized in Table III,
where column 2 shows the upper bound of Lemma 1,
and column 3 is the upper bound $\lfloor{(n+8)/4}\rfloor$
on the minimum degree vertex $x$ in the $K_4$ under consideration.

\vspace{7mm}
\begin{center}
\begin{tabular}{|r|cc|cc|}
\hline 
$n$ \T \B & $n+8$ & $\lfloor{n/4}\rfloor+2$ & $|V(F^+_x)|$ & $|V(F^-_x)|$ \\
\hline
25 & 33 & 8 & 7 & 17 \\
 & & & 8 & 16 \\
\hline
24 & 32 & 8 & 6 & 17 \\
 & & & 7 & 16 \\
 & & & 8 & 15 \\
\hline
23 & 31 & 7 & 5 & 17 \\
 & & & 6 & 16 \\
 & & & 7 & 15 \\
\hline
22 & 30 & 7 & 4 & 17 \\
 & & & 5 & 16 \\
 & & & 6 & 15 \\
 & & & 7 & 14 \\
\hline
  \end{tabular}
\end{center}

\begin{center}
{\bf Table III.}
Possible parameters of $(K_5-P_3,K_5;n)$-good\\
graphs containing $K_4$, for $n \ge 22$.
\end{center}
\vspace{7mm}

Let $F$ be a $(K_5-P_3,K_5;25)$-good graph containing $K_4$.
By Table I and Table II,
there are only 3 possible graphs for $F^-_x$ and 18 possible graphs
(with a $K_3$) for $F^+_x$. The computation we ran determined
the possible ways these graphs can be connected.
Given a vertex $v$ in $F^+_x$, we define the {\em cone of $v$}
as the set of all the vertices adjacent to $v$ in $F^-_x$.
There are many restrictions we can place on these cones
with the given parameters:

\begin{description}
 \item[(C1)] {\em The cones of any two vertices in any $K_3$ in $F^+_x$ must be disjoint.} Otherwise, the vertices of this $K_3$, $x$ and any vertex in the intersection of two of these cones will create a $K_5-P_3$.
 \item[(C2)] {\em The complement in $V(F^-_x)$ of the union of the cones of any two non-adjacent vertices, $a$ and $b$, must not contain an independent set of order 3.} Otherwise this independent set together with $a$ and $b$ will be an independent set of order 5.
 \item[(C3)] {\em The complement in $V(F^-_x)$ of the cones of any three non-adjacent vertices, $a$, $b$ and $c$, must not contain an independent set of order 2 (that is, it must be complete).} Otherwise we again have an independent set of order 5 with the vertices $a$, $b$, $c$, and any two non-adjacent vertices in the complement.
 \item[(C4)] {\em The intersection of the cones of any two adjacent vertices, $a$ and $b$, must not contain an edge.} Otherwise the vertices $a$, $b$, $x$ and the vertices connected by this edge will create a $K_5-P_3$.
\end{description}

These constraints are not exhaustive enough to fully characterize
$(K_5-P_3,K_5)$-good graphs, but they are sufficiently
restrictive that we will be able to prove that
$R(K_5-P_3,K_5) = 25$ and find the sole $(K_5-P_3,K_5;22)$-good graph.

\section{Computation of $R(K_5-P_3,K_5)$} \label{sec:thm}

\medskip
As shown in Figure \ref{fig:chain}, all possible neighborhoods of $x$ in $V(F)$ on 25 vertices can be constructed from a triangle and two independent vertices.  The highly similar substructure of $F^+_x$ will allow us to eliminate constructions without having to attempt to arrange all $7$ or $8$ cones.  For the three $(K_5-P_3,K_4)$-good graphs on $16$ and $17$ vertices we find arrangements of cones on five vertices, either a triangle and two non-adjacent vertices or a triangle and two adjacent vertices.  The first three vertices form a triangle and its cones are subject to condition \textbf{(C1)}.  If the last two vertices are adjacent their cones must satisfy condition \textbf{(C4)}, otherwise \textbf{(C2)} .  The last two vertices are also independent to each vertex in the triangle, and must satisfy condition \textbf{(C2)} with these vertices as well.  Finally, if the last two vertices are independent they, along with any of the vertices in the triangle, must satisfy condition \textbf{(C3)}. \\

\begin{figure}[h]
\begin{center}
\includegraphics[scale=.5]{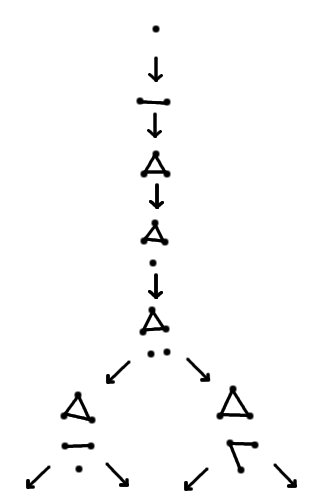}
\caption{Two constructions of $F^+_x$ on up to 6 vertices
form a subgraph of all graphs induced by the apexes of
cones needed in computations. All $F^+_x$ must contain
$K_3$ with two isolated vertices.
}
\label{fig:chain}
\end{center}
\end{figure}

\begin{theorem} \label{thm:main} {\ }
\begin{description}
\item[(1)] $R(K_5-P_3,K_5)=25$.
\item[(2)] There are no $(K_5-P_3,K_5;22)$-good graph containing a $K_4$ on 23 or 24 vertices, and there is a unique $(K_5-P_3,K_5;22)$-good graph which contains a $K_4$.
\end{description}
\end{theorem}

\begin{proof}
The proof is computational.
\begin{description}
\item[(1)] For the three $(K_5-P_3,K_4)$-good graphs on $|V(F^-_x)|=16$ and $17$ vertices there were no valid arrangements of five cones. In fact, there is no valid arrangement of 4 cones for the graph on $17$ vertices. Therefore the result follows.
\item[(2)] When we run the algorithm building cone arrangements
for $|V(F^-_x)|=15$ we find
that there is no valid arrangement of 6 cones.
This eliminates the possibility of creating a $(K_5-P_3,K_5;n)$-good
graph with a $K_4$ for $n=23,24,25$.
For $|V(F^-_x)|=14$ we find that there is one valid arrangement of 7 cones.
This is for $F^+_x = C_3 \cup C_4$, and it gives us exactly
one $(K_5-P_3,K_5;22)$-good graph with a $K_4$,
whose adjacency matrix is given
in Figure \ref{fig:adjmat}.
\end{description}
\end{proof}

\begin{figure}[h]
\begin{verbatim}
         1  0 1 1 1 1 1 1 1 0 0 0 0 0 0 0 0 0 0 0 0 0 0
         2  1 0 1 1 0 0 0 0 1 1 1 1 0 0 0 0 0 0 0 0 0 0
         3  1 1 0 1 0 0 0 0 0 0 0 0 1 1 1 1 1 0 0 0 0 0
         4  1 1 1 0 0 0 0 0 0 0 0 0 0 0 0 0 0 1 1 1 1 1
         5  1 0 0 0 0 1 1 0 1 0 1 0 1 1 0 0 0 0 1 0 1 0
         6  1 0 0 0 1 0 0 1 0 1 0 1 0 1 0 1 0 1 1 0 0 0
         7  1 0 0 0 1 0 0 1 0 1 1 0 1 0 0 0 1 0 0 1 0 1
         8  1 0 0 0 0 1 1 0 1 0 0 1 0 0 1 0 1 1 0 0 0 1
         9  0 1 0 0 1 0 0 1 0 1 0 1 1 0 1 0 0 0 0 1 1 0
        10  0 1 0 0 0 1 1 0 1 0 1 0 0 0 1 1 0 1 0 1 0 0
        11  0 1 0 0 1 0 1 0 0 1 0 1 0 0 0 1 1 0 1 0 0 1
        12  0 1 0 0 0 1 0 1 1 0 1 0 0 1 0 0 1 0 0 0 1 1
        13  0 0 1 0 1 0 1 0 1 0 0 0 0 1 1 0 0 1 1 0 0 1
        14  0 0 1 0 1 1 0 0 0 0 0 1 1 0 0 1 0 1 0 1 0 1
        15  0 0 1 0 0 0 0 1 1 1 0 0 1 0 0 1 0 0 1 0 1 1
        16  0 0 1 0 0 1 0 0 0 1 1 0 0 1 1 0 0 0 0 1 1 1
        17  0 0 1 0 0 0 1 1 0 0 1 1 0 0 0 0 0 1 1 1 1 0
        18  0 0 0 1 0 1 0 1 0 1 0 0 1 1 0 0 1 0 1 1 0 0
        19  0 0 0 1 1 1 0 0 0 0 1 0 1 0 1 0 1 1 0 0 1 0
        20  0 0 0 1 0 0 1 0 1 1 0 0 0 1 0 1 1 1 0 0 1 0
        21  0 0 0 1 1 0 0 0 1 0 0 1 0 0 1 1 1 0 1 1 0 0
        22  0 0 0 1 0 0 1 1 0 0 1 1 1 1 1 1 0 0 0 0 0 0
\end{verbatim}
\caption{The unique $(K_5-P_3,K_5;22)$-good graph with a $K_4$.
Vertices 1 through 4 form $K_4$, $x$ is the first vertex,
vertices 5 through 8 induce $C_4$,
and vertices 9 through 22 are those in $V(F^-_x)$}
\label{fig:adjmat}
\end{figure}

The main computations were performed at least twice
with independent implementations by the first two authors.
They agreed on the number of possible cone arrangements
in all cases and on the final results.

\section{Some Related Ramsey Numbers}

Burr, Erd\H {o}s, Faudree and Schelp \cite{BEFS} proved
a theorem showing that certain small extensions of complete graphs don't increase the Ramsey number.  Let $\widehat{K}_{n,p}$ be the unique graph obtained by connecting a new vertex $v$ to $p$ vertices of a $K_n$.

\medskip
\bigskip
\begin{theorem} \label{thm:erdos} \emph{\cite{BEFS}}
For $m,n \ge 3$ and $m+n \ge 8$,
\[
R(\widehat{K}_{m,p},\widehat{K}_{n,q})=R(K_m,K_n)
\]
\[
\text{with }p=\left\lceil \frac{m}{n-1} \right\rceil \text{ and } q=\left\lceil \frac{n}{m-1} \right\rceil.
\]
\end{theorem}

\noindent Note that this theorem also implies $R(\widehat{K}_{m,p},K_n)=R(K_m,K_n)$.  For the case $m=4$ and $n=5$, this theorem shows that $R(\widehat{K}_{4,1},\widehat{K}_{5,2}) = 25$, which does not prove $R(\widehat{K}_{4,2},K_5) = R(K_5-P_3,K_5) = 25$.  However, using Theorem \ref{thm:main} and slightly modifying the proof of Theorem \ref{thm:erdos}, we can further show that $R(\widehat{K}_{4,2},\widehat{K}_{5,2}) = 25$.

\medskip
\bigskip
\begin{theorem} \label{thm:list}
All of the following Ramsey numbers are equal to $25$.
\begin{description}
\item[(1)] $R(K_4,\widehat{K}_{5,2})$,
\item[(2)] $R(\widehat{K}_{4,1},K_5)$,
\item[(3)] $R(\widehat{K}_{4,1},\widehat{K}_{5,2})$,
\item[(4)] $R(\widehat{K}_{4,2},K_5)$,
\item[(5)] $R(\widehat{K}_{4,2},\widehat{K}_{5,2})$.
\end{description}
\end{theorem}

\medskip
\bigskip
\begin{proof}
\begin{description}
\item[(1-3)] Directly from Theorem \ref{thm:erdos}. 
\item[(4)] Proved in Section \ref{sec:thm}. 
\item[(5)] Take a $(\widehat{K}_{4,2},\widehat{K}_{5,2})$-good coloring of a $K_{25}$. Then there must be a blue $K_5$. By Theorem \ref{thm:erdos} and the fact that $R(K_3,K_5)=14$, we have that $R(\widehat{K}_{3,1},\widehat{K}_{5,2})=14$.  Therefore, in the 20 vertices not contained in the blue $K_5$ there must be a red $\widehat{K}_{3,1}$. Each vertex of the blue $K_5$ is adjacent in blue to at least one vertex of the red $\widehat{K}_{3,1}$. So some vertex of the red $\widehat{K}_{3,1}$ is adjacent in blue to at least 2 vertices of the blue $K_5$, creating a blue $\widehat{K}_{5,2}$.
\end{description}
\end{proof}

\bigskip
To conclude, we note that the difficulty of Theorem 3.4,
namely the title case of this paper, is apparently far greater
than that of all other cases covered by Theorem 3. Better understanding
of this difference could lead to an improvement of Theorem 2
covering all manageable small extensions of complete graphs.

\bigskip
\noindent
{\bf Note.}
Recently, Boza \cite{Boza} obtained the equality
$R(K_5-P_3,K_5)=25$ using an approach very different
from ours.

\bigskip
\noindent
{\bf Acknowledgement.} We would like to thank an anonymous
reviewer for numerous suggestions of how to improve
the presentation of this paper.

\end{document}